\documentclass[11pt]{amsart}
\usepackage{amsmath,amssymb}
\usepackage{graphics,graphicx,epsfig}
\usepackage[width=16cm,height=23cm]{geometry}
\theoremstyle{plain}

\usepackage{graphicx}
\usepackage{graphics,epsfig}
\usepackage{float}
\DeclareGraphicsExtensions{.bmp,.png,.pdf,.jpg,.jpeg}
\usepackage[ansinew]{inputenc}
\usepackage{float}
\usepackage[usenames,dvipsnames,svgnames,table]{xcolor}
\begingroup

\newtheorem{thm}{Theorem}[section]
\newtheorem{cor}[thm]{Corollary}
\newtheorem{lemma}[thm]{Lemma}
\newtheorem{prop}[thm]{Proposition}

\endgroup
\begingroup

\theoremstyle{definition}

\endgroup

\begingroup

\theoremstyle{definition}

\newtheorem{ex}[thm]{Example}

\endgroup

\begingroup
\theoremstyle{remark}
\newtheorem{rem}[thm]{Remark}

\endgroup

\numberwithin{equation}{section}

\newcommand{\re}{\mathbb{R}}
\newcommand{\nat}{\mathbb{N}}

\newcommand{\rw}{\rightarrow}
\newcommand{\lin}{{L^{\infty}(\mu)}}
\newcommand{\lpv}{{L^{p(\cdot)}(\mu)}}

\newcommand{\lqv}{{L^{q(\cdot)}(\mu)}}
\newcommand{\lpvo}{{L^{p(\cdot)}[0,1]}}
\newcommand{\lqvo}{{L^{q(\cdot)}[0,1]}}

\begin{document}
\setlength{\baselineskip}{6mm}
\title[Disjointly strictly singular inclusions between variable Lebesgue spaces]{Disjointly strictly singular inclusions between variable Lebesgue spaces}

\author[Hern\'andez, Ruiz and Sanchiz]{ Francisco L. Hern\'andez$^*$,\ C{\'e}sar Ruiz$^*$ and Mauro Sanchiz$^{**}$}
\address{IMI and
 Departamento de An{\'a}lisis Matem{\'a}tico y Matem\'atica Aplicada, Facultad de
Matem{\'a}ticas, Universidad Complutense, 28040 Madrid, Spain}
\email{
(F.L. Hern\'andez)  pacoh@ucm.es}
\email{
(C. Ruiz)  cruizb@mat.ucm.es}
  \email{
  (M. Sanchiz) mauro.sanchizalonso@ceu.es}
\thanks{$^*$ Partially supported  by grant PID2019-107701G-I00}
\thanks{$^{**}$ Partially supported by grant PID2019-107701G-I00 and scholarship CT42/18-CT43/18}

\begin{abstract} Disjointly strictly  singular inclusions between variable Lebesgue spaces $\lpv$ on finite measure are characterized. Suitable criteria in terms of the (bounded or unbounded) exponents are given. It is proved the equivalence of  $L$-weak compactness (also called almost compactness) and disjoint strict singularity  for variable Lebesgue space inclusions. For infinite measure any inclusion \,$\lpv \hookrightarrow \lqv$\,  is not disjointly strictly singular. No restrictions on the exponent are imposed.

\end{abstract}

\maketitle 

\section{Introduction}  A linear operator $T$ between two Banach spaces $E$ and $F$ is \textit{strictly singular} (or Kato) if $T$ fails to be an isomorphism on any infinite-dimensional
(closed) subspace of  $E$, i.e. given $\epsilon >0$ and an infinite dimensional subspace $E_{0}$ of  $E$ there exists an unitary  vector $x\in E_{0} $ such that  $\|Tx\| \leq \epsilon$. In the context of Banach lattices $E$ a useful weaker notion is that  an operator $T$ from  $E$ to  $F$  is said to be \textit{disjointly strictly singular} (DSS in short) if there is no disjoint sequence of non-null vectors $(x_{n})$ in $E$ such that the restriction of $T$ to the (closed) subspace \, $[x_{n}]$\, spanned by $(x_{n})$ is an isomorphism.


The study of  strictly and disjointly strictly  singular inclusions have  been quite extensive for symmetric (or rearrangement invariant) function spaces. Recall that for symmetric function spaces \,$E(\mu)$\, on finite measures  the  left canonic inclusions of \,$L^{\infty}(\mu)$  in \,$E(\mu)$\, is always  strictly singular, while the right inclusion  of \,$ E(\mu)$  in \,$L^{1}(\mu)$\,  is disjointly strictly singular. And  this inclusion \,$ E(\mu)\hookrightarrow L^{1}(\mu)$\, is strictly singular if and only if the Orlicz space  \,$L^{\exp x^{2}}_{0}$\, cannot be   included in \,$E(\mu)$. When considering two symmetric function spaces with \,$E(\mu)  \hookrightarrow F(\mu)$\,   this  inclusion $i$
is strictly singular if and only if \,$i$\, is disjointly strictly singular   and the norms of  \,$E(\mu)$\, and \,$F(\mu)$\, are not equivalent on \,$[r_{n}]_{E(\mu)}$\, and $[r_{n}]_{F(\mu)}$, the subspaces spanned by the Rademacher functions $(r_{n})$ (\cite{A-H-S,F-H-K-T,H-N-S}). This strengthens the interest in knowing characterizations of disjointly strictly singular inclusions  for distinguished  classes of function spaces  (see \cite{A-S,G-H-S-S,H-R-S-12} and references within).


One of the  goals of this paper is to study the  disjoint strict singularity  of inclusion operators between   variable  Lebesgue  spaces (or Nakano spaces)\, $\lpv$\,\, for finite and infinite measures.
These non-symmetric classical function spaces  \,$\lpv$\, have seen a strong renewed relevance  in the last decades due to their applications (cf.  \cite{Libro,Libro2}).

 In this context of variable spaces the inclusion behavior is more diverse than in the symmetric case. Compact and weakly compact inclusions have been  considered in \cite{G-M,F-G-N-R,E-G-N,H-R-S-21}.  The study of $L$-weak compactness  of variable space inclusions  is motivated by its applications to the compactness of associated  Sobolev  embeddings (see \cite{E-G-N,F-G-N-R}). Recall that an operator $T$ between two Banach function lattices $E$ and $F$ on a measure space $(\Omega,\mu)$  is said to be \emph{$L$-weakly compact} (or \emph{almost compact}) whenever $T(B_E)$ is a equi-integrable subset in \,$F$\, for $B_{E}$ \,denoting the unit ball of $E$, i.e.
 $$ \lim_{n\rw \infty} \sup_{f\in B_E} \{|| Tf\chi_{A_n}||_F\} =0,
 $$ for every sequence \,$(A_n)$\,  of measurable sets in $\Omega$ with  \,$\chi_{A_n}\rw 0$ $\mu$-a.e.. In \cite{E-G-N} (Thm. 3.4)
     Edmunds, Gogathisvili  and Nekvinda have given the following  $L$-weak compactness criterion for bounded exponents defined on bounded open subsets \,$\Omega$\, of \,$\re^n$ with Lebesgue measure $\vert \cdot\vert$: an inclusion $L^{p(\cdot)}(\Omega) \hookrightarrow L^{q(\cdot)}(\Omega) $\,
   is  $L$-weakly compact  if and only if for every \,$a>1$,
    $$ \int_{0}^{|\Omega|} a^{\left(\frac{1}{p-q}\right)^*(x)} dx < \infty,
    $$
    where $(\frac{1}{p-q})^*(x)$ denotes the decreasing rearrangement of $(\frac{1}{p-q})(t)$. Another  $L$-weak compactness criterion in $\lpv$ (of De la Valle\'e-Pousin type) has been given  in \cite{H-R-S-21} (Prop. 3.3).

The study of disjointly strictly singular inclusions \,$\lpv \hookrightarrow \lqv$\,  was initiated in  \cite{F-H-R-S}. In the present  paper we continue this  research line obtaining now complete characterizations of  disjointly strictly singular inclusions  and  $L$-weakly compact inclusions  \,$\lpv \hookrightarrow \lqv $, giving suitable conditions on the exponents. It comes out the equivalence  of these two  concepts in this setting of variable Lebesgue  space inclusions (a fact rather unexpected according with the Orlicz space behavior, see Section 2). The strict singularity of inclusions  \,$L^{\infty}(\mu) \hookrightarrow \lpv$\, is also studied giving suitable criteria for it.


The paper is divided in 6 sections. Section 2 recall some definitions and basic results. Section 3  contains some useful preliminary results on $\lpv$ spaces and decreasing rearrangement functions. Thus Proposition \ref{prop-33}  states, by an analysis of disjoint function sequences \,$(\frac{\chi_{E_n}}{\mu(E_n)^{\frac{1}{p(t)}}})$,  that  if an inclusion \,$ \lpv \hookrightarrow \lqv $\,  is disjointly strictly singular,  then
 $$ \lim_{x\rw \mu(\Omega)^{-}} \,(\mu(\Omega)-x)^{(\frac{p-q}{p\,q})^*(x)} =   0. $$

In Section 4, disjointly strictly singular inclusions $\lpv  \hookrightarrow \lqv$  for    \textit{finite}  measures are studied, looking for suitable criteria  on the exponents. First we do under the  hypothesis  of  the exponent \,$q(\cdot)$\, be bounded (Theorem \ref{Teor-2}). After that we consider the general case, thus Theorem \ref{Teor-2b}  claims the equivalence of the following statements for exponents \, $q(\cdot) < p(\cdot) $ $\mu$-a.e. on a finite measure space:

\begin{enumerate}
  \item The inclusion \,$ \lpv \hookrightarrow \lqv$\,  is $L$-weakly compact.
  \item The inclusion $ \lpv \hookrightarrow \lqv$ \,is disjointly strictly singular.
\vspace{1mm}
  \item $\lim_{x\rw \mu(\Omega)^{-}} \, \,(\mu(\Omega)-x)^{(\frac{p-q}{p\,q})^*(x)}= 0$.
\vspace{1mm}
  \item $\int_0^{\mu(\Omega)} a^{(\frac{p\,q}{p-q})^*(x)}dx < \infty$ \,\,\, for every $a>1$.
\end{enumerate}


Thus the above $L$-weak compactness inclusion criteria for bounded exponents in \cite{E-G-N} is extended to the general case. The useful limit condition (3) has not been considered earlier.
In particular a new weak compactness criterion for inclusions  $ \lpv \hookrightarrow L^{1}(\mu)$ is given (Corollary \ref{Cor3}).
The  strict singularity of inclusions \,$L^{\infty}(\mu) \hookrightarrow \lpv$\, is also studied obtaining the following criterion   $$\lim_{x\rw \mu(\Omega)^-} (\mu(\Omega) -x) ^{(\frac{1}{p})^*(x)} =\, 0,$$ which is   equivalent to \,$\int_0^{\mu(\Omega)} a^{p^*(x)} dx < \infty$\, for every $a>1$.   In other words, the exponent \,$p(\cdot)$\,  must belong to the Orlicz space \,$ L_{0}^{exp \,x}(\mu)$ (Theorem \ref{Teor-4}).
  Several  illustrative examples are included at the end of this section (Examples at \ref{examples}).

   In Section 5, the special exponent class of  log-Holder continuous functions is considered,  giving a simpler disjoint strict singularity  criterion, namely
   \, $$ess\inf (p(\cdot)-q(\cdot)) > 0\,.$$

   Finally,   Section 6 is devoted to the \textit{infinite} measure case. Inclusions \, $\lpv \hookrightarrow \lqv$ \,   for  infinite measures \,forces  that the exponents have  a very close asymptotic behavior. This allows  to find suitable  subspaces generated by disjoint functions with equivalence of norms. Thus for infinite measures all the  inclusions   \,$ \lpv \hookrightarrow  \lqv $\,  are  no  disjointly strictly singular (Theorem \ref{DSS-Inf3}).


\section{Preliminaries}

We recall here some basic definitions and fix the notation used in the following sections.


 An  operator \, $T : E \rw Y$ between a Banach  lattice $E$ and a Banach space  $F$ is \textit{disjointly strictly singular}\,(DSS in short)  if the restriction $T|_{[f_{n}]}$ \,is not an isomorphism for any (closed)  subspace $[f_n]$ \,spanned by a normalized pairwise disjoint sequence $(f_{n}) $ in  $E$.
 This DSS notion is useful in comparing the lattice structure of Banach lattices and studying strictly singular operators between Banach lattices  (cf. \cite{H-S,F-H-K-T}). Recall that an operator $T$ between two Banach spaces $E$ and $F$ is \textit{strictly singular }(or Kato) if there is no infinite-dimensional subspace $E_{1}$ of $E$ such that the restriction $T_{|E_{1}}$ is an isomorphism. Obviously  every strictly singular operator is DSS but  the converse is not true. (f.i.  the inclusions \,$L^{p}[0,1] \hookrightarrow L^{q}[0,1] $, $q < p<\infty$).

An operator  \,$T : E \rw F$ between a Banach function lattice  \,$E$ and a Banach space $F$ is said to be \emph{$M$-weakly compact }  whenever \,$\lim_{n\rw\infty} \,||T(f_n)||_F = 0$,  where $(f_n)$ is any norm bounded disjoint sequence in \,$E$.
 It is clear that every  $M$-weakly compact operator is a DSS  operator.
An  operator \,$T : E \rw F$ between two Banach function lattices  $E$ and $F$ on a measure space $(\Omega,\mu)$  is \emph{$L$-weakly compact (or almost compact or strict)}   whenever $T(B_E)$ is a  equi-integrable subset in \,$F$\, for $B_{E}$ \,the unit ball of $E$ \,i.e.
 $$ \lim_{n\rw \infty} \sup_{f\in B_E} \{||T(f)\chi_{A_n}||_F\} =0\,
 $$ for every sequence \,$(A_n)$\,  of measurable sets in $\Omega$ such that \,$\chi_{A_n}\rw 0$ $\mu$-a.e. (cf. \cite{A-B,B-S}).

Let $(\Omega, \mu)$ be a  measure space. Given an \emph{exponent} function \,$p(\cdot)$\, on $\Omega$ (i.e. a real measurable function $p$ on $ \Omega$ with $1\leq p(t) < \infty $) the   \emph{variable   Lebesgue  space (or Nakano space)} \, $\lpv$\, is  the space of all  real  measurable function classes  $f$ on $\Omega$\, such that the modular \,$\rho_{p(\cdot)} (f/r) < \infty$ for some  \,$r >0$, where
  $$\rho_{p(\cdot)}(f) = \int_{\Omega} | f(t)|^{p(t)} d\mu. $$

The associated Luxemburg norm is defined by
$$ ||f||_{p(\cdot)} = \inf \{r>0 : \rho_{p(\cdot)} \left(\frac{f}{r}\right) \leq 1\}.
$$

We denote  \, $p^-:=\text{ess}\inf\{p(t):t\in\Omega\}$\, and $p^+:=\text{ess}\sup\{p(t):t\in\Omega\}$.
Equally, $p_{\vert A}^+$ and $p_{\vert A}^-$ denote the essential supremum and infimum of \,$p(\cdot)$\, over a measurable subset $A$ of $\Omega$.
When $\Omega = \nat$ with the counting measure and $(p_{n})$ is a  real sequence with $1\leq p_{n}< \infty$, we get the {\em  Nakano sequence space} $\ell_{(p_n)}$ i.e.  the Banach space
$$\ell_{(p_n)}=\left\{(x_n)\in\mathbb R^{\mathbb N}:\rho((x_n))=\sum_{n=1}^\infty\left|\frac{x_n}r\right|^{p_n}\,<\infty\mbox{ for some }r>0\right\}$$
equipped with the corresponding Luxemburg norm.


The \textit{conjugate exponent} function \,$p'(\cdot)$\, of $p(\cdot)$ is defined by the equation $\frac{1}{p(t)}+\frac{1}{p'(t)}= 1$ almost everywhere $t\in \Omega$.
When $p^+<\infty$, the topological dual of the space $\lpv$\,   is the variable Lebesgue  space $L^{p'(\cdot)}(\Omega)$. An space $\lpv$ is separable  if and only if the measure space $(\Omega, \mu)$ is separable and $p^+ < \infty$. Moreover, $\lpv$ is reflexive if and only if \,$1< p^- \leq p^+ <\infty$.

Recall that the \textit{associated} space \,$(\lpv )'$\, is the space of all scalar measurable functions $g$  on $\Omega$ such that  $ \int_{\Omega} \,f g \,d\mu < \infty$\, for every \,$f\in\lpv$. If  $ 1<p(\cdot) < \infty$ a.e. then $(L^{p(\cdot)}(\mu))'= L^{p'(\cdot)}(\mu)$ \, (cf. \cite{Libro2,M-W}). 


A sequence $(f_{n}) \subset  \lpv$ verifies that  $||f_n||_{p(\cdot)}\rw 0$ if and only if  $ \rho_{p(\cdot)}( \lambda f_n) \rw 0$ for every $\lambda > 0$.  If  $p^+< \infty$, then $\rho_{p(\cdot)}( f_n) \rw 0$  if and only if $||f_n||_{p(\cdot)}\rw 0$. Furthermore,  \,$\rho_{p(\cdot)}( f) \leq1$ if and only if \, $||f||_{p(\cdot)}\leq1$. Also, if $||f||_{p(\cdot)}> 1$, then $1\leq ||f||_{p(\cdot)}\leq \rho_{p(\cdot)}(f)$ (\cite{Libro} p.75, \cite{Libro2}).
 The \textit{H\"older inequality } (\cite{Libro2} Thm 2.26, \cite{Libro} Lemma 3.2.20) states that there exists a constant \,  $1< K \leq 4$\, such that for every two measurable functions  \,$f,g : \Omega \rw \re$, it holds
        $$ \int_{\Omega} |f(t)g(t)| d\mu \leq \,K \,||f||_{p(\cdot)}\, ||g||_{p'(\cdot)}.
    $$

A criterion for the inclusion \,$\lpv \hookrightarrow \lqv $\, to hold is the following:

\begin{prop}\label{inclusion} (\cite{Libro2} Thm 2.45, \cite{Libro} Thm 3.3.1)
Let $(\Omega,\mu)$ be an atomless infinite measure space and exponents $p(\cdot)$ and $q(\cdot)$. The inclusion \,$\lpv \hookrightarrow \lqv$\, holds if and only if \,$q(\cdot)\leq p(\cdot)$\, $\mu$-a.e. and there exists $\lambda>1$ such that
$$
\int_{\Omega_d} \lambda^{-\left(\frac{p\, q}{p-q}\right)(t)}\,  d\mu <\infty,
$$
where \,$\Omega_d=\{t\in\Omega : p(t)>q(t)\}$.
\end{prop}
Note that, in contrast with classical Lebesgue spaces \,$L^{p}$, inclusions between variable Lebesgue spaces on infinite measures \,$\lpv \hookrightarrow \lqv$\, can hold. Also, for a finite measure space $(\Omega,\mu)$, the inclusion \,$\lpv \hookrightarrow \lqv $\, holds if and only if \,$q(\cdot) \leq p(\cdot)$\, $\mu$-a.e..

If  $(f_n)$ is  a disjoint sequence in $\lpv$ and  $(g_n)$ is another sequence such that
    $ \sum ||f_n-g_n||_{p(\cdot)} <\infty$ then $(f_n)$ and $(g_n)$ are equivalent (unconditional) basic sequences, i.e.
    $ \sum_{n} x_nf_n \in \lpv$ if and only if $\sum_{n} x_ng_n \in \lpv.$
\begin{prop}$\text{(\cite{N})}$\label{lemaNa}
Let $1\leq p_{n},\,q_{n} <\infty$.  Then \,$\ell_{(p_{n})} =  \ell_{(q_{n})} $\, if and only if   there exists  \,$\alpha > 0 $\,  such that  \,$ \sum_{n=1}^{\infty}\,  \alpha^{\frac{p_{n}q_{n}}{|p_{n}-q_{n}|}}\, < \infty $.
\end{prop}


Recall that the \emph{decreasing rearrangement} (cf. \cite{B-S,K-P-S,LT2}) of a measurable function $f$ is the real function $f^*$ on $[0, \mu(\Omega))$ defined by
$$ f^*(x) := \inf \{s \in [0,\mu(\Omega)] \,:\,\mu_f(s) \leq x\},$$
where $\mu_f$ is the distribution function of $f$,
$\mu_f(s) := \mu(\{t\in \Omega : |f(t)| > s\})$.
For  a measurable $f \geq 0$ on $\Omega$, the functions $f$ and $f^*$ are equi-distributed and
$$ \int_{\Omega} f(t) d\mu \,= \, \int_0^{\mu(\Omega)} f^*(x) dx. $$

 A Banach function lattice  is said to be \emph{rearrangement invariant} if every two equi-distributed functions have the same norm. Orlicz spaces (cf. \cite{K-R,Mu}) are examples of rearrangement invariant spaces  while variable Lebesgue spaces are not. If $\varphi$ is a non-decreasing unbounded positive convex function on $[0,\infty)$ with  $\varphi(0)=0 $, the \emph{Orlicz space} \,$L^{\varphi}(\mu)$\, consists of all measurable functions classes $f$ on $(\Omega,\mu)$ such that for some  $r>0$
$$\int_{\Omega} \varphi(r |f|)d\mu  <\infty .$$

In the class of Orlicz spaces there are examples of inclusions  \,$L^{\varphi}(\mu) \hookrightarrow L^{\psi}(\mu)$\, for a finite measure which are DSS but not $L$-weakly compact.
Let us  recall the  DSS  criterion and the $L$-weak compactness criterion for inclusions between  Orlicz spaces:


    \begin{prop}(\cite{H-S} Prop. 3.2). Let $(\Omega,\mu)$ be an atomless finite measure space  and \,$\psi \leq \varphi $ \,Orlicz functions with the $\Delta_{2}$-condition. An inclusion \,$L^{\varphi}(\mu)\hookrightarrow L^{\psi}(\mu) $ \, is  DSS  \, if and only if for every natural $n$ and any constant $A>0$ there exist \,$1\leq x_{1} < x_{2}< ...< x_{n}$\, and  $c_{i} > 0$ for  $i=1,...,n$ such that for  $t \geq 1$
    $$\sum_{i=1}^{n}  c_{i} \,\psi(tx_{i}) \, \leq \, A  \,  \sum_{i=1}^{n}  \,c_{i} \, \varphi(tx_{i})$$
    \end{prop}

   \begin{prop}\label{lem-0}(\cite{K-R}, \cite{A-S} p.1369) Let \,$(\Omega,\mu)$\,  be an atomless finite measure space and Orlicz functions \,$\psi \leq  \varphi $. An inclusion \,$L^{\varphi}(\mu)\hookrightarrow L^{\psi}(\mu) $\, is \,$L$-weakly compact if and only if
        $$  \lim_{t\mapsto \infty}  \,\frac{\varphi^{-1}(t)}{\psi^{-1}(t)}\, =\, 0. $$
   \end{prop}

In particular this  condition implies  that \,   $ \lim_{s\mapsto\infty} \, \frac{\psi(s)}{\varphi(s)} = 0$.
Indeed, given  $0<\epsilon <1$  there exists $s_{0}>0$ such that  \, $\varphi^{-1}(s) \leq \epsilon \psi^{-1}(s)$\, for   $s>s_{0}$. So, for \, $s=\psi(t)>s_{0}$, we have  $\varphi^{-1}(\psi(t)) \leq \epsilon \psi^{-1}(\psi(t)) = \epsilon $, hence \, $\psi(t) \leq  \varphi(\epsilon t)\, \leq \epsilon\,\varphi(t)$ \,  for\,  $t > \varphi^{-1}(s_{0})$.


Consider now  the Orlicz function  $\psi$  defined  in (\cite{H-S2} Thm. A) which verifies
the inclusion \,$L^{p}[0,1] \hookrightarrow L^{\psi}[0,1]$, for a fixed $p>1$. Using the above criterion, it is proved  that the inclusion \,  $L^{p}[0,1] \hookrightarrow L^{\psi}[0,1]$ \, is DSS and  that
$$ \limsup_{t\mapsto \infty}\,  \,\frac{\psi(t)}{t^{p}}\,   \geq  \limsup_{n\mapsto\infty}\,\, \frac{\psi(2^{n})}{2^{np}}\, >  0
$$
(see \cite{H-S2} p.184). So we deduce, by above $L$-weak compactness criterion, that the inclusion  \, $L^{p}[0,1] \hookrightarrow L^{\psi}[0,1]$  \, is not $L$-weakly compact.


    In the setting of variable Lebesgue spaces  $\lpv$,   a  $L$-weakly  compact inclusion criterion has been given by Edmunds, Gogathisvili  and Nekvinda in (\cite{E-G-N} Thm 3.4):

 \begin{prop}\label{edm}(\cite{E-G-N})
  Let a bounded open subset \,$\Omega \subset \re^n$\
   and  bounded exponents  \,$ q(\cdot) \leq p(\cdot) \, \leq \,p^+ < \infty$. The inclusion \,
   $ L^{p(\cdot)}(\Omega) \hookrightarrow L^{q(\cdot)}(\Omega)$\, is $L$-weakly compact if and only if for every \,$a>1$
    $$ \int_{0}^{|\Omega|} a^{\left(\frac{1}{p-q}\right)^*(x)} dx < \infty.
    $$

  \end{prop}



Other $L$-weak  compactness inclusion  characterization of \, $L^{p(\cdot)}(\mu) \hookrightarrow L^{q(\cdot)}(\mu)$ (De la Valle\'e-Poussin type) is given in \cite{H-R-S-21} (Prop. 3.3), \cite{Tesis}.





\section{Previous Results}

We   will  study  inclusion operators between variable Lebesgue spaces on a finite measure space \,$\lpv \hookrightarrow \lqv $, looking for suitable DSS characterizations in terms of the exponents. In this section we collect some preliminary results.
We  will assume \,$q(\cdot) \leq \, p(\cdot)$ $\mu$-a.e. with \,$\mu\{t\in \Omega:\,q(t) = p(t)\} = 0$\,  (otherwise the inclusion $ \lpv \hookrightarrow \lqv$\, would be trivially non-DSS). No restrictions  on  the    exponents will be assumed.


Let  $r : \Omega \rw [1,\infty)$ be an $exponent$  and consider the function \, $ a ^{r(t)}$\, on $\Omega$ for some  $a>1$. It  holds  that \, $ (a^{r(\cdot)})^*(x) = \,a^{r^{*}(x)}$ for all $x>0$   (cf. \cite{E-G-N}  Lemma 2.10), hence
$$ \int_{\Omega} a^{r(t)} d\mu = \int_0^{\mu(\Omega)} (a^{r(\cdot)})^*(x)\, dx = \int_0^{\mu(\Omega)} a^{r^*(x)} dx.
$$


\begin{lemma} \label{lema-1} Let \,$(\Omega,\mu)$ be an atomless finite measure space and $r(\cdot)$ be an exponent. If there exists \, $a>1$\, such that \,  $\int_{0}^{\mu(\Omega)} a^{r^{*}(x)} dx =\infty$, then there exist \, $\beta > 0$ and a disjoint measurable sequence \,   $(E_n)_{n=1}^{\infty}$ \, such that, for every natural $n$,
$$ ||\chi_{E_n}||_{r(\cdot)} \geq \, \beta.
$$
\end{lemma}

\begin{proof}
Let  $0<\beta <1$ \,  with  \,  $1/\beta > a $. Since the rearrangement  $r^{*}(\cdot)$ is decreasing and $\int_0^{\mu(\Omega)} a^{r^*(x)} dx =\infty$\,, there exists a positive  sequence  $(t_{n})$, with  $t_{n}\searrow 0$ \, and  $t_{1}= \mu(\Omega)$ , such that
$$\int_{t_{n+1}}^{t_{n}} \, a^{r^{*}(x)}\, dx > 1.
$$
Now, for each  $t_{n} >0$ we can find a measurable set $F_{n} \subset \Omega$ with \,$\mu(F_{n})= t_{n}$\, such that
$$
\int_{F_n}  \, a^{r(t)}\,d\mu\, =  \,\int_{0}^{t_{n}} (a^{r(\cdot)})^{*}(x) \, dx = \, \int_{0}^{t_{n}} \, a^{r^*(x)} \,dx.
$$
Moreover, the sets $(F_{n})$ can be defined so that  $F_{n+1} \subset F_{n}$ \, since $t_{n+1}< t_n$\, (cf. \cite{B-S} Lemma 2.2.5).

Consider now the  disjoint measurable sequence  $(E_{n})$ where \, $E_{n} := F_{n} \backslash  F_{n+1}$. Then,

\begin{align*}
\int_{\Omega} (\frac{\chi_{E_{n}}}{\beta})^{r(t)} \, d\mu 
\geq{}  &   \int_{E_{n}} (a \chi_{E_{n}})^{r(t)}  d\mu \, = \int_{F_{n}} (a \chi_{F_{n}})^{r(t)}  d\mu \,-  \int_{F_{n+1}} (a \chi_{F_{n+1}})^{r(t)}  d\mu \\
={} &    \int_{0}^{t_{n}} a^{r^*(x)}\,dx \, - \int_{0}^{t_{n+1}}\,a^{r^*(x)} dx =  \int_{t_{n+1}}^{t_{n}} a^{r^*(x)}\,dx\,\, > \,\, 1
\end{align*}
for every natural $n$. Hence, $||\chi_{E_n}||_{r(\cdot)} \geq \, \beta$.
\end{proof}




\begin{prop} \label{Teo-1} Let $(\Omega, \mu)$ be an atomless finite measure space and an exponent \,$r(\cdot)$. Then every sequence $(E_n)_{n=1}^{\infty}$ in \,$\Omega$  with $\chi_{E_n} \rw 0$ $\mu$-a.e. satisfies \,$ ||\chi_{E_n}||_{r(\cdot)} \rw 0  $ \, if and only if for every \, $a>1$,
$$ \int_0^{\mu(\Omega)} a^{r^*(x)} dx < \infty. $$
\end{prop}

\begin{proof}
The direct implication is above Lemma \ref{lema-1}. Let us show the converse.
Assume there exist a measurable  sequence  $(E_n )$   with \,$\chi_{E_n} \rw 0$ $\mu$-a.e. (thus \,$\mu(E_n)\rw 0$\,) and
\, $0 < \delta < 1$ \, such that $||\chi_{E_n}||_{r(\cdot)} \geq \delta$ for every $n$.
Taking \,$0 <\beta < \delta$, it follows from the norm definition that
$$ 1 < \int_{\Omega} (\frac{\chi_{E_n}}{\beta})^{r(t)} d\mu.
$$
Now, by the  hypotheses,
$$ \int_{\Omega} (\frac{\chi_{\Omega}}{\beta})^{r(t)} d\mu = \int_{0}^{\mu(\Omega)} (\frac{1}{\beta})^{r^*(x)} dx <\infty
$$
and, as $ \frac{\chi_{E_n}}{\beta} \rw 0 \, \mu-\text{a.e.}$, we conclude using the dominated convergence theorem that
$$ \int_{\Omega} (\frac{\chi_{E_n}}{\beta})^{r(t)} d\mu \rw 0  \,
$$  as $n \mapsto \infty$, which is a  contradiction.
\end{proof}

The above equivalence, for bounded exponents, was crucial for proving Theorem 3.4 in \cite{E-G-N} (see also  \cite{E-L-N}).
 \begin{prop} \label{prop-33} Let \,$(\Omega, \mu)$ \,be an atomless  finite measure space and exponents\,  $q(\cdot) \leq p(\cdot)$. If the inclusion \,$ \lpv \hookrightarrow \lqv $\,  is DSS, then
     $$ \lim_{x\rw \mu(\Omega)^{-}} \,(\mu(\Omega)-x)^{(\frac{p-q}{p\, q})^*(x)} =   0. $$
  \end{prop}

\begin{proof}
Since the inclusion is DSS we have \,$\mu(\{t:\,p(t)=q(t)\})=0$. First note that if \,$ess \inf (\frac{p-q}{pq})^* > 0$, then $\lim_{x\rw \mu(\Omega)^{-}} (\mu(\Omega)-x)^{(\frac{p-q}{pq})^*(x)} = 0$. Hence (since the rearrangement is decreasing), we can suppose
     $$ \lim_{x\rw \mu(\Omega)^{-}} \,(\frac{p-q}{p\, q})^*(x) = 0.
     $$

     Let us assume  that \,  $ \limsup_{x\rw \mu(\Omega)^{-}} (\mu(\Omega)-x)^{(\frac{p-q}{p\,q})^*(x)}  > 0 $. Then there exist
      $r>0$ and a scalar sequence \, $(x_n)\nearrow \mu(\Omega)$ \,  such that
\begin{equation}{\tag{$\diamond$}}
    (\mu(\Omega)-x_n)^{(\frac{p-q}{p\,q})^*(x_n)}\ge \, r
\end{equation}
     for every natural\, $n$. Furthermore  it can be assumed w.l.o.g. that  \, $\frac{x_n +\mu(\Omega)}{2} < x_{n+1} $ \, and
     $$(\frac{p-q}{p\,q})^*(\frac{x_n+\mu(\Omega)}{2})> (\frac{p-q}{p\,q})^* (x_{n+1})$$ for every  natural $n$.
     Consider the sets
     $$ A_n= \big(\,(\frac{p-q}{p\, q})^*\big)^{-1} \left(\,[(\frac{p-q}{p\, q})^*(\frac{x_n+\mu(\Omega)}{2}),(\frac{p-q}{p\, q})^*(x_n)]\right)\supseteqq \;[x_n, \frac{x_n+\mu(\Omega)}{2}].
     $$
     Thus   $(A_n)$ is a disjoint measurable sequence with Lebesgue measure 
     \begin{equation}{\tag{$\diamond\diamond$}}
         |A_n|\,\geq \,\frac{\mu(\Omega)-x_n}{2}.
     \end{equation}

     Let us define the sets
     $$ B_n=\{ t\in \Omega \,\,:\,\,  (\frac{p-q}{p\,q})^*(\frac{x_n+\mu(\Omega)}{2})\leq \frac{p-q}{p\,q}(t)\leq (\frac{p-q}{p\,q})^*(x_n)\}.
     $$
     Since the functions \,$\frac{p-q}{p\, q}$ and  $(\frac{p-q}{p\, q})^*$ are  equi-distributed we have \,$\mu(B_n)= |A_n|$. The sets $(B_n)$ are disjoint and  it can be assumed \, $0<\mu(B_n)<1$ for every $n$.
     Consider the disjoint normalized sequence \,$(s_{n})$\, in $\lpv$
     $$ s_n(t) := \frac{\chi_{B_n}}{\mu(B_n)^{\frac{1}{p(t)}}}
     $$  and the (closed) subspace \,$[s_n]_{p(\cdot)}$. Note that \,$\sum_n a_ns_n \in [s_n]_{p(\cdot)}$ \,if and only if $$\rho_{p(\cdot)}(\lambda \sum_{n>N} a_n s_n) \xrightarrow{{N\rw \infty}} 0 \quad \text{for every}\quad \lambda >0.$$

     Let us prove that $i|_{[s_n]_{p(\cdot)}}$ is an isomorphism showing that \,$(s_n)$ and $(is_n)$ are equivalent basic sequences. Since $i$ is continuous, we only need to show that $\sum y_ns_n \in [s_n]_{q(\cdot)} $ implies  $\sum y_ns_n \in [s_n]_{p(\cdot)}.$
     First notice that
     $$\rho_{p(\cdot)}\left(\sum \lambda  y_n s_n \right) = \sum \int_{B_n} |\lambda y_n|^{p(t)}\frac{\chi_{B_n}}{\mu(B_n)}d\mu=
      \sum \int_{B_n} |\lambda y_n|^{q(t)}|\lambda y_n|^{p-q(t)}\frac{\chi_{B_n}}{|A_n|^{\frac{q(t)}{p(t)}}|A_n|^{1-\frac{q(t)}{p(t)}}}d\mu.
     $$
     Also, $|\lambda y_n|<1$ up to a finite amount of terms for every $\lambda>0$. Otherwise, since for large enough $N$ we have $\rho_{q(\cdot)}(\lambda\sum_{n> N} y_ns_n)< \infty$, taking $\lambda_0 > \frac{2}{r}$  we get

     \begin{align*}
    \rho_{q(\cdot)} (\sum_{n>N} \lambda_0\lambda \,y_n s_n) \geq{}  &    \sum_{n>N} \int_{B_n} \lambda_0^{q(t)}\frac{1}{ |A_n|^{\frac{q(t)}{p(t)}}}d\mu =
          \sum_{n>N} \frac{1}{ |A_n|}\int_{B_n} \lambda_0^{q(t)}|A_n|^{1 -\frac{q(t)}{p(t)}}d\mu \\
          ={}   &   \sum_{n>N} \frac{1}{ |A_n|}\int_{B_n}\left[\lambda_0 |A_n|^{\frac{p(t)-q(t)}{q(t)p(t)}}\right]^{q(t)} d\mu
      \intertext{and using $(\diamond\diamond)$ and $(\diamond)$ we have}
      \geq{}    &   \sum_{n>N} \frac{1}{ |A_n|}\int_{B_n} \left[\lambda_0 (\frac{\mu(\Omega)-x_n}{2})^{\frac{p(t)-q(t)}{q(t)p(t)}}\right]^{q(t)}d\mu \\
      \geq{}    &   \sum_{n>N} \frac{1}{ |A_n|}\int_{B_n} \left[\lambda_0 (\frac{\mu(\Omega)-x_n}{2})^{(\frac{p(t)-q(t)}{q(t)p(t)})_{|_{B_n}}^{+}}\right]^{q(t)}d\mu \\
      \geq{}    &   \sum_{n>N} \frac{1}{ |A_n|}\int_{B_n} \left[\lambda_0 (\frac{\mu(\Omega)-x_n}{2})^{(\frac{p-q}{qp})^{*}(x_n)}\right]^{q(t)}d\mu \\
      \geq{}    &   \sum_{n>N} \frac{1}{ |A_n|}\int_{B_n} [\lambda_0 r\frac{1}{2}]^{q(t)}d\mu \,
           =\, \infty,
     \end{align*}
    which  is a contradiction.

    Now, using this fact and writing \,$\lambda= \frac{\lambda'}{\lambda_0}$, we have

 \begin{align*}
  \rho_{p(\cdot)} (\sum_{n>N} \lambda y_ns_n)&=\sum_{n>N} \int_{B_n} |\frac{\lambda'}{\lambda_0} y_n|^{q(t)}|\frac{\lambda'}{\lambda_0} y_n|^{p-q(t)}\frac{\chi_{B_n}}{|A_n|^{\frac{q(t)}{p(t)}}|A_n|^{1-\frac{q(t)}{p(t)}}}d\mu\\
    &\leq
    \sum_{n>N} \int_{B_n} |\lambda' y_n|^{q(t)}\frac{\chi_{B_n}}{|A_n|^{\frac{q(t)}{p(t)}}}\left[\frac{1}{\lambda_0(|A_n|)^{\frac{p(t)-q(t)}{p(t)q(t)}}}\right]^{q(t)}d\mu \\
     &\leq \sum_{n>N} \int_{B_n} |\lambda' y_n|^{q(t)}\frac{\chi_{B_n}}{|A_n|^{\frac{q(t)}{p(t)}}}\left[\frac{1}{\lambda_0(|A_n|)^{(\frac{p(t)-q(t)}{p(t)q(t)})^{+}_{|_{B_n}}}}\right]^{q(t)}d\mu
     \\
  &\leq \sum_{n>N}  \int_{B_n}| \lambda' y_n|^{q(t)}\frac{\chi_{B_n}}{|A_n|^{\frac{q(t)}{p(t)}}}\left[\frac{1}{\lambda_0(\frac{\mu(\Omega)-x_n}{2})^{(\frac{p-q}{pq})^*(x_n)}}\right]^{q(t)}d\mu\\
& \leq \sum_{n>N}  \int_{B_n}| \lambda' y_n|^{q(t)}\frac{\chi_{B_n}}{|A_n|^{\frac{q(t)}{p(t)}}}\left[\frac{1}{\lambda_0 r(\frac{1}{2})^{(\frac{p-q}{pq})^*(x_n)}}\right]^{q(t)}d\mu\\
   &\leq \rho_{q(\cdot)}(\sum_{n>N} \lambda' y_ns_n) \xrightarrow{N\rw \infty} 0.
  \end{align*}
  This  concludes the proof.
\end{proof}

 \begin{cor} \label{lem-3} Let \,$(\Omega, \mu)$ \,be an atomless  finite measure space and exponents\,  $q(\cdot) < p(\cdot)$ \,$\mu$-a.e. If the inclusion \,$ \lpv \hookrightarrow \lqv $\,  is DSS, then
     $$ \lim_{x\rw \mu(\Omega)^{-}} \,(\mu(\Omega)-x)^{(\frac{p-q}{p})^*(x)} =   0.$$
Moreover, if $p^{+} < \infty$,  then  $$ \lim_{x\rw \mu(\Omega)^{-}}\, (\mu(\Omega)-x)^{(p-q)^*(x)} = 0.
     $$
  \end{cor}

\begin{proof}
    If follows from above proposition and  when $p^{+} < \infty$ using that   \, $(\frac{p-q}{p^{+}})(\cdot)\leq (\frac{p-q}{p})(\cdot) \leq (\frac{p-q}{p^{-}})(\cdot)$.
\end{proof}

The converse of the above proposition will be proved later. We will need some  basic Lemmas:



\begin{lemma}\label{lema-a} Let  \,$(\Omega,\mu) $\,be  a
finite measure space and   $f: \Omega \rw (0,\infty)$ be a measurable function. Then the functions $1/f$ and $1/f^*$ are  equi-distributed. Hence \, $(1/f)^* \,= \,(1/f^*)^*$.
\end{lemma}

\begin{proof}
    Let $\lambda > 0$. Since 
\begin{align*}
    \mu_{1/f}(\lambda)\, ={}    &   \mu(\{t\in \Omega \,:  1/f(t) > \lambda\,\})\, \, =\, \mu(\{t\in \Omega \,\,:\,\, 1/\lambda > f(t)\,\}) \\
    ={} &   \mu(\Omega) - \mu(\{t\in \Omega :\, f(t) \geq 1/\lambda\,\})= \mu(\Omega)-|\{t\in [0,\mu(\Omega)] \,: \, f^*(t) \geq 1/\lambda\,\}| \\
    ={} &   |\,\{t\in [0,\mu(\Omega)] \,:\, f^*(t)< 1/\lambda\,\}|= |\{t\in [0,\mu(\Omega)] \,:\, 1/f^*(t) > \lambda\,\}| = |\,\,|_{1/f^*}(\lambda),
\end{align*}
the functions  \,$1/f$ and  $1/f^*$ are  equi-distributed.
\end{proof}

\begin{lemma}\label{lema-2} If $f: [0,\mu(\Omega)]  \rw (0,\infty)$ is an increasing measurable function then   \, $f^*(t)= f(\mu(\Omega) -t)$.
\end{lemma}

\begin{proof}
Let $\lambda > 0$ and consider $s_{\lambda}: = \inf\{ t\in [0,\mu(\Omega)] \,:\, f(t) > \lambda \,\}$. Since  $f$  is increasing we have \,
$ \mu_{f} (\lambda) = \mu(\Omega) -s_{\lambda}
$. On the other hand, the function \,$h(t) := f(\mu(\Omega) -t)$ is decreasing and
$$ \{ t \in [0,\mu(\Omega)] \,\,:\,\, h(t) > \lambda\,\} = (0\,,\mu(\Omega)-s_{\lambda}).$$
\end{proof}

\begin{lemma} \label{prop-1} Let \,$f:[0,b) \rw (0,\infty) $ \, be  a decreasing measurable function with \,$f(0)>0$ \,  and \, $\lim_{x\rw b^{-}}\,f(x)= \,0$.
If \,\,$\lim_{x\rw b^{-}}\, (b-x)^{f(x)} = 0 $, then,  for every \, $a>1$,
$$ \int_{0}^{b} a^{\frac{1}{f(x)}}\,dx\, < \,\infty.
$$
\end{lemma}

\begin{proof}
Given $a>1$, consider a natural  $N $ such that \,$a< e^N$. Let  us see  $\int_{0}^{b} \,e^{N/f(x)}\,dx < \infty$.

From the hypothesis it follows that \,  $\lim_{x\rw b^{-}}-f(x) \ln (b-x) = \infty$\,. Hence there exists $0<\delta_N <1$ such that \,
$ -f(x) \ln(b-x) \geq 2N $ \,for every $x\in (b-\delta_N, b)$, i.e. $\, \frac{1}{2N} \geq \frac{1}{-f(x) \ln(b-x)}$. Thus
$$ \int_{0}^{b}\, a^{1/f(x)}dx \leq \, \int_{0}^{b- \delta_N} a^{1/f(x)}dx + \int_{b- \delta_N}^{b} (e^N)^{1/f(x)}dx \, <\, \infty,$$  since

\begin{align*}
    \int_{b- \delta_N}^{b} (e^N)^{1/f(x)}dx ={} &    \int_{b- \delta_N}^{b} \left(e^N \right)^{\frac{-\ln(b-x)}{-f(x)\ln(b-x)}} dx \leq \int_{b- \delta_N}^{b} \left(e \right)^{-\ln(b-x)^N(\frac{1}{2N})} dx \\
    ={}  &   \int_{b- \delta_N}^{b}\left(\frac{1}{(b-x)^N}\right)^{\frac{1}{2N}}dx = \int_{b- \delta_N}^{b}\frac{1}{\sqrt{b-x}}dx \,\,< \infty.
\end{align*}
\end{proof}


\section{The finite measure case}
In this section we give suitable criteria in terms of the exponents for the inclusions between  variable exponent Lebesgue spaces  over finite measure spaces be DSS. First we consider the case of inclusions  $\lpv  \hookrightarrow \lqv$ when  $q(\cdot)$ is a bounded exponent. After that, we will do the general case.

In particular we get  the equivalence of the $L$-weak compactness and the DSS property  for inclusions between  variable Lebesgue spaces (recall that this equivalence does not happen  in  Orlicz spaces, see  Section 2).


\begin{thm} \label{Teor-2} Let $(\Omega,\mu) $ be an atomless  finite measure space and
 exponents \,$q(\cdot) < p(\cdot)$ $\mu$-a.e. with $q^+ < \infty$. Denote $i$ the inclusion $i:\lpv \hookrightarrow \lqv$. TFAE:
\begin{enumerate}
\item $\int_0^{\mu(\Omega)} a^{(\frac{p}{p-q})^*(x)}dx < \infty$ \,\,\, for every $a>1$.
\item $i$ is $L$-weakly compact.
\item $i$ is $M$-weakly compact.
\item $i$ is $DSS$.
\item The restriction of the inclusion $i$ on  any subspace spanned by a  disjoint sequence \,$(\frac{\chi_{E_n}}{\mu(E_n)^{\frac{1}{p(t)}}})$\, is not an isomorphism.
\item $\lim_{x\rw \mu(\Omega)^{-}} \, \,(\mu(\Omega)-x)^{(\frac{p-q}{p})^*(x)}= 0$.
\end{enumerate}
\end{thm}

\begin{proof}
$(1) \Rightarrow (2)$ It is similar to Lemma 3.3 in \cite{E-G-N}. Suppose that $(1)$ is true and $(2)$ is not. Hence
$$ \lim_{n\rw \infty} \sup_{f\in B_{\lpv}} \{||f\chi_{A_n}||_{q(\cdot)}\} \neq 0
 $$ for certain sequence $(A_n)$ in  $\Omega$\, such that $\chi_{A_n}\rw 0$ $\mu$-a.e.. Then there exist \, $0< \delta <1$, a sequence  $(f_k)_k\subset B_{\lpv}$ and a subsequence  $(A_{n_k})_k$ such that
 $$ || f_k\chi_{A_{n_k}}||_{q(\cdot)} \, > \, \delta
 $$ for every natural  $k$. This implies
     $$1 < \|\frac{f_{k}\chi_{A_{n_{k}}}}{\delta}\|_{q(\cdot)} \leq \rho_{q(\cdot)}(\frac{f_{k}\chi_{A_{n_{k}}}}{\delta})\, \leq \, \frac{1}{\delta^{q^{+}}}\, \rho_{q(\cdot)}(f_{k}\chi_{A_{n_{k}}}).$$
 Hence,
 $$   \delta^{q^+}   \, <   \, \rho_{q(\cdot)}( f_k\chi_{A_{n_k}}) = \int_{\Omega} | f_k \chi_{A_{n_k}}(t)|^{q(t)} d\mu
 $$
 for every $k$. Now considering the exponent\,  $r(t) := \frac{p(t)}{q(t)}>1$\, $\mu$-a.e. and its conjugate \,$r'(t) = \frac{p(t)}{p(t)-q(t)}$ we have, by H\"older inequality, that
 $$  \rho_{q(\cdot)}( f_k\chi_{A_{n_k}}) \leq \, K \, ||\, \chi_{An_k}||_{r'(\cdot)} \, \, ||  f_k^{q(\cdot)}||_{r(\cdot)}.$$
 And, since \, $||\,f_k\,||_{p(\cdot)} \leq 1$, we have \, $||\,f_k^{q(\cdot)}\,||_{r(\cdot)} \leq 1$. Indeed, if $\rho_{p(\cdot)}(\frac{f}{\lambda}) \leq 1$ for every $\lambda >1$, then
 $$ \rho_{r(\cdot)}\left(\frac{f^q}{\lambda}\right) \leq \int_{\Omega} \frac{|f(t)|^{p(t)}}{(\lambda^{1/q^+})^{p(t)}} d\mu \leq 1.
 $$
 Thus,
 $$ \delta^{q^+}   \leq  \,K \, ||\chi_{An_k}||_{r'(\cdot)}.$$
 But by hypotheses $(1)$ and Proposition \ref{Teo-1} we get \,$||\chi_{An_k}||_{r'(\cdot)} \rw 0$, which is a contradiction.

$(2)\Rightarrow (3)$ Let  \,$(f_{n})$\, be a pairwise disjoint normalized sequence in  $\lpv$. As \,$\mu(\Omega) < \infty$\, we have \,$\mu(supp(f_{n}))\rightarrow 0$ as $n\rightarrow \infty$, hence
      $$\lim_{n\rightarrow \infty} \|f_{n}\|_{q(\cdot)} \leq \lim_{n\mapsto \infty} \sup_{k} \,\|f_{k}\chi_{supp(f_{n})}\|_{q(\cdot)} \,\leq \,\lim_{\mu(A)\rightarrow 0} \sup_{k} \, \|f_{k} \chi_{A} \|_{q(\cdot)} =  0.  $$

 It is clear  that $(3) \Rightarrow (4) \Rightarrow (5)$.

    $(5) \Rightarrow (6)$ It is  the proof of Proposition \ref{prop-33} (and Corollary \ref{lem-3})

    $(6) \Rightarrow   (1)$  Consider  the function \,$f(t) = (\frac{p-q}{p})(t)$. From Lemma \ref{lema-a}, we have that the functions \,  $\frac{1}{(\frac{p-q}{p})^*(x)} $ \, and  \, $\frac{1}{(\frac{p-q}{p})(t)}=(\frac{p}{p-q})(t) $ \,are equi-distributed. Thus \, $(\frac{p}{p-q})^* = (1/f^*)^*$\, and
$$
\int_{0}^{\mu(\Omega)} a^{\left(\frac{p}{p-q}\right)^*(x)}dx  = \int_{0}^{\mu(\Omega)} a^{\left(\frac{1}{(\frac{p-q}{p})^*}\right)^*(x)}dx,
$$
 Now, as $1/f^*$ is an increasing function,  it follows by  Lemma \ref{lema-2} that  \,$(1/f)^*(x)= (1/f^*)^*(x) = (1/f^*)(\mu(\Omega)-x)$\, so we have
$$
 = \,\int_{0}^{\mu(\Omega)} a^{\frac{1}{(\frac{p-q}{p})^*}(\mu(\Omega)-x)}dx  \,=\,\int_{0}^{\mu(\Omega)} a^{\frac{1}{(\frac{p-q}{p})^*(y)}}dy.$$

Finally the  boundedness of this integral follows  from  Lemma  \ref{prop-1} for  $\mu(\Omega) = b $,  since \, $(\frac{p-q}{p })^*(\cdot)$ is decreasing and
the hypothesis.
\end{proof}

 \, It is clear that under the hypotheses of \;$ ess\inf (p-q) > 0$ \, and \,$p^+<\infty$ we have $$\lim_{x\rw \mu(\Omega)^{-}}(\mu(\Omega)-x)^{(\frac{p-q}{p})^*(x)}= 0,$$
 so the inclusion $\lpv \hookrightarrow \lqv$ is DSS (theses hypotheses are not necessary conditions, see Example \ref{examples}).

For variable Lebesgue spaces on  finite measures \,$\lpv$\, we have the canonical extreme inclusions \,$L^{\infty}(\mu)\hookrightarrow \lpv \hookrightarrow L^{1}(\mu)$\, (cf. \cite{Libro,Libro2}).

If follows from the above a DSS criterion for  the  right  extreme  case of \,$ L^{1}(\mu) \equiv \lqv$  (recall that  weakly compact subsets in \,$L^{1}(\mu)$\, are the same as equi-integrable sets by Dunford-Pettis Theorem, cf. \cite{A-K}):

\begin{cor}\label{Cor3}

Let $(\Omega,\mu)$ be an atomless finite measure space and an exponent  $p(\cdot)$. TFAE:
\begin{enumerate}
\item \,$\int_0^{\mu(\Omega)} a^{(\frac{p}{p-1})^{*}(x)} \,dx < \infty$ \, \, for every $a>1$.
\item The inclusion \, $\lpv \hookrightarrow L^{1}(\mu)$ \, is weakly compact.
\item The inclusion \, $\lpv \hookrightarrow L^{1}(\mu)$ \, is DSS.
\item $\lim_{x\rw \mu(\Omega)^{-}} (\mu(\Omega)-x)^{(\frac{p-1}{p})^*(x)} = 0$.
\end{enumerate}
\end{cor}


\begin{rem}
Notice that the above inclusions  $i:\lpv \hookrightarrow L^1(\mu)$, are not strictly singular.
\end{rem} 

Indeed, in  the constant exponent case $p_{-}= p_{+}$ it is well-known (Khintchine inequality, cf. \cite{LT1} p.66). Assume now  \, $p_{-}\leq r < p_{+} \leq \infty$ and  consider the set \,$\Omega_{r} = \{t\in \Omega : p(t)<r  \}$, which have  \,$\mu(\Omega_{r})> 0$.  The  restricted variable Lebesgue space \,$L^{p(\cdot)}(\Omega_{r}) $\, can be  canonically identified with a closed band-subspace of \,$\lpv$. Now as \,$p|_{\Omega_{r}}^{+} \leq r < \infty$\, it follows easily from Khintchine inequalities in $L^{p}$-spaces  (cf. \cite{A,LT1})\, that the Rademacher function system \,$(r_{n})$\,  in \, $L^{p(\cdot)}(\Omega_{r}) $\, and  in $L^{1}(\mu)$\, are  equivalent  to the canonical basis of \,$\ell_{2}$. Hence the inclusion  \,  $\lpv \hookrightarrow  L^{1}(\mu) $\, is  not strictly singular.

A similar argument shows also  that all the inclusions \,$\lpv \hookrightarrow \lqv$\, (for  $ q(\cdot) \leq p(\cdot)$) are not strictly singular.


We consider now  the left extreme inclusions  \,$ L^{\infty}(\mu)\hookrightarrow \lpv$. In this case the  DSS property is  equivalent to the strict singularity:

\begin{thm} \label{Teor-3} Let $(\Omega,\mu)$ be an atomless  finite measure space and an exponent \,$p(t)$. Denote $i$ the inclusion $i: L^{\infty}( \mu) \hookrightarrow \lpv$. TFAE:
\begin{enumerate}
  \item $\int_0^{\mu(\Omega)} a^{p^*(x)} dx < \infty$ \,\,\, for every $a>1$.
  \item   Every  measurable set sequence $(E_n)_n$   with $ \chi_{E_n} \rw 0$ a.e. satisfies  $||\chi_{E_n}||_{p(\cdot)} \rw 0$.
  \item $i$ is $L$-weakly compact.
  \item $i$ is $M$-weakly compact.
  \item $i$ is weakly compact.
  \item $i$ is strictly singular.
  \item $i$ is DSS.
\end{enumerate}
\end{thm}

\begin{proof}
$(1)\Leftrightarrow (2)$. It is Proposition 3.2.

\noindent \textbf{$(2) \Rightarrow (3).$}  If $f\in \lin$, as $||f||_{p(\cdot)} \leq \left\lVert (||f||_{\infty})\chi_{\Omega}\right\rVert_{p(\cdot)} = ||f||_{\infty} || \chi_{\Omega}||_{p(\cdot)}$, we have
$$ \lim_{\mu(A)\rw 0} \sup \{||f\chi_A||_{p(\cdot)} \,\,:\,\,||f||_{\infty}\leq 1\,\} \leq \lim_{\mu(A)\rw 0} ||\chi_A||_{p(\cdot)} = 0.
$$

\noindent \textbf{$(3) \Rightarrow (4).$} It is clear.

\noindent  \textbf{$(4) \Leftrightarrow (5)$}. Since $\lin$ is an $AM$-space (cf. \cite{A-B} Thm. 18.11),

\noindent  \textbf{$(5) \Rightarrow (6).$} Since  $\lin$ has the Dunford-Pettis property and the inclusion $i$  is weakly compact (cf. \cite{Go} Thm. 3.3.5).

\noindent   \textbf{$(6) \Rightarrow (7).$} It is obvious.

\noindent \textbf{$(7) \Rightarrow (1).$} Assume that there exists \, $a>1$\, such that $\int_0^{\mu(\Omega)} a^{p^*(x)} dx = \infty$. Then, by  Lemma  \ref{lema-1}, there exists \, $\beta > 0$\,  and a sequence of disjoint measurable sets \, $(E_n)_n$\,  verifying \,  $||\chi_{E_n}||_{p(\cdot)}\geq \beta$ \,for every $n$. Now  the basic sequence $(\chi_{E_n})_n$ in $L^{\infty}( \mu)$ and in $\lpv$ are equivalent. Indeed, since the inclusion \,$i : L^{\infty}( \mu) \hookrightarrow \lpv$\, is continuous, we have  \, $||\sum_{n=1}^{\infty}a_n\chi_{E_n}||_{p(\cdot)} \, \leq C ||\sum_{n=1}^{\infty}a_n\chi_{E_n}||_{\infty} $ \, for some $C>0$. And for every natural $k$ we have
    $$
    ||\sum_{n=1}^{k}a_n\chi_{E_n}||_{\infty}= \max_{1\leq n\leq k} |a_n| \,\leq \,\,   \frac{1}{\beta} \,\, || \sum_{n=1}^{k}a_n\chi_{E_n}||_{p(\cdot)}.$$
This contradicts the DSS property.
\end{proof}

A direct consequence, by factorization, is the following:
\begin{cor} Let $(\Omega,\mu)$ be an atomless  finite measure space and exponents \,$q(t) \leq p(t)$ $\mu$-a.e.. If
$$ \int_0^{\mu(\Omega)} a^{q^{*}(x)} dx = \infty \qquad \text{for some } a>1,
$$ then the inclusion $\lpv \hookrightarrow \lqv$ is not DSS.
\end{cor}

 Let us give another strictly singular criteria for inclusions \,$L^{\infty}( \mu) \hookrightarrow \lpv$.
 Recall  that the  \textit{exponential}  Orlicz space \,$L^{exp}_{0}(\mu) $ \, defined  by the function \,$ \varphi(x) = e^{x}-1 $\, and a finite measure is the space of all measurable functions $f$ such that
$$ \,\, \int_{\Omega} e^{| r f(t)|} \,d\mu < \infty \quad\text{ for every } r > 0 .
$$
Then, we have:

\begin{thm}\label{Teor-4}  Let $(\Omega,\mu)$ be an atomless finite measure space and  an exponent $p(t)$. TFAE:
\begin{enumerate}
  \item The inclusion  \,$\displaystyle L^{\infty}( \mu) \hookrightarrow \lpv$ \, is strictly singular.
  \item \,$\displaystyle  p(\cdot) \in \,L^{exp}_{0}(\mu)$.
  \item  $\displaystyle \, \lim_{x\rw 0}\, \frac{\int_0^x p^*(s) ds}{x\ln (\frac{e}{x})}\,= \, 0$.
  \item  $\displaystyle \lim_{x\rw \mu(\Omega)^-} (\mu(\Omega) -x) ^{(\frac{1}{p})^*(x)} =\, 0 .$

\end{enumerate}
\end{thm}

\begin{proof}
$(1) \Leftrightarrow (2)$. Assume  that the inclusion $\displaystyle  L^{\infty}( \mu) \hookrightarrow \lpv$  is strictly singular. Then, by the above characterization,
$$ \int_0^{\mu(\Omega)} a^{p^*(x)} dx =  \int_{\Omega} a^{p(t)} d\mu = \int_0^{\mu(\Omega)} (a^{p(\cdot)})^*(x) \,dx   < \infty
$$
for every $a>1.$ Since $1<a = e^{\frac{1}{s}}$ for some  \,$0 < s < \infty$,   we have
$$  \int_{\Omega} a^{p(t)} d\mu =  \int_{\Omega} e^{\frac{p(t)}{s}} d\mu <\infty
$$ for every $s>0$. Therefore $p(\cdot)\in L^{exp}_{0}(\mu)$. The converse is equal.

The equivalence $(2) \Leftrightarrow (3)$ follows from the following known fact: the order continuous exponential Orlicz space \,$ L^{exp}_{0}(\Omega)$\, coincides with the order-continuous Marcinkiwiecz space \,$M_0(\varphi)$\, defined  by the
function \, $\varphi(x)= x \ln (\frac{e}{x})$\, (cf. \cite{K-P-S} p.116). Hence the exponent  $p(\cdot) \in L^{exp}_{0}(\Omega)$\, satisfies the  condition
$$\lim_{x\rw 0} \, \frac{\int_0^x p^*(s) ds}{x\,\ln (\frac{e}{x})}\, = \, 0.
$$

$(1) \Rightarrow (4)$.  First assume $1<p(t)$ \, $\mu$-a.e.. Then, since \,$1<p'(t) = \frac{p(t)}{p(t)-1}< \infty$, $(1)$ says that $\int_0^{\mu(\Omega)} a^{(\frac{p'}{p'-1})^*(x)}dx < \infty $ for every $a>1$. And Corollary \ref{Cor3} allows to get
$$ \lim_{x\rw \mu(\Omega)^-} (\mu(\Omega) -x) ^{(\frac{p'-1}{p'})^*(x)} =\, 0 .
$$
Assume now that $\Omega_1 = \{t\in \Omega : p(t)=1 \}$ has positive measure,  take then the exponent \,$r(t) = 2\chi_{\Omega_{1}}+p(t) \chi_{\Omega_{1}^{c}}$.
Thus
$$ \int_{0}^{\mu(\Omega)} a^{r^*(x)}dx = \int_{\Omega} a^{r(t)}d\mu \leq \int_{\Omega_1} a^2 d\mu + \int_{\Omega\backslash \Omega_1}a^{p(t)} d\mu < \infty.
$$ Now, since $\displaystyle \lim_{x\rw \mu(\Omega)^-} (\mu(\Omega) -x) ^{(\frac{1}{r})^*(x)} =\, 0$ and $(\frac{1}{r})^*(\cdot)\leq (\frac{1}{p})^*(\cdot)$,  we conclude that
$$ \displaystyle \lim_{x\rw \mu(\Omega)^-} (\mu(\Omega) -x) ^{(\frac{1}{p})^*(x)} =\, 0.
$$

\noindent $(4)\Rightarrow (1)$ The case $p^+ < \infty$ is trivial. Assume now that $p^+ = \infty.$ Since $(\frac{1}{p})^*$ is decreasing and \,  $\lim_{x\rw \mu(\Omega)} (\frac{1}{p})^*(x)=0$, it follows from Lemma \ref{prop-1} that
$$ \int_0^{\mu(\Omega)} a^{\frac{1}{(\frac{1}{p})^*}(x)} dx < \infty \qquad \text{for every } a>1 .
$$
Now, using Lemma \ref{lema-a}, the functions \,$p = \frac{1}{\frac{1}{p}}$ and $\frac{1}{(\frac{1}{p})^*}$ are equi-distributed, thus
$$ \int_0^{\mu(\Omega)} a^{p^*(x)} dx = \int_0^{\mu(\Omega)} a^{(\frac{1}{(\frac{1}{p})^*})^*(x)} dx =  \int_0^{\mu(\Omega)} a^{\frac{1}{(\frac{1}{p})^*}(x)} dx < \infty.
$$
\end{proof}


We pass now to study the general case of DSS inclusions \,$\lpv \hookrightarrow \lqv$\, for unbounded exponents. First notice that now the condition (6) in above Theorem 4.1 is no enough for getting disjoint strict singularity, as the following example shows:

\begin{ex} \label{ejemplo-0} Let $q(x)= \frac{1}{x}$ with $x\in (0,1)$  and \,$p(x) = (1+\epsilon)q(x)$ \, for some  $\epsilon > 0$. Since $\int_0^1 a^{\frac{1}{x}} dx = \infty$ for every $a>1$, we have  by Corollary 4.4 that  the inclusion
$ L^{p(\cdot)} \hookrightarrow L^{q(\cdot)}$\, is not DSS,  in spite of
$$ \lim_{x\rw 1^-} (1-x)^{(p-q/p)^*(x)}=  \lim_{x\rw 1^-} (1-x)^{\epsilon/1+\epsilon}\,=\, 0.
$$
\end{ex}


We wonder now what happens with the inclusions $\lpv \hookrightarrow \lqv$ \, whether  $q^+=\infty$ and $\displaystyle \int_0^{\mu(\Omega)} a^{q^*(x)} dx < \infty$ for every $a>1$. Notice that, if for $a>1$
$$
\int_0^{\mu(\Omega)} a^{(\frac{p\,q}{p-q})^*(x)}dx < \infty,
$$
then also $ \int_0^{\mu(\Omega)} a^{q^*(x)}dx < \infty$. Indeed, for \,$1<f(t)=\frac{p}{q}(t) $ we have
$$
\frac{p\,q}{p-q}(t)= \frac{f(t)}{f(t)-1}\,q(t) > q(t),
$$
so \,$(\frac{p\,q}{p-q})^*(x)\geq q^*(x)$.  This leads to the following:


\begin{thm} \label{Teor-2b} Let $(\Omega,\mu) $ be an atomless finite measure space and exponents \, $q(\cdot) < p(\cdot) $ $\mu$-a.e.. TFAE:
\begin{enumerate}
    \item $\int_0^{\mu(\Omega)} a^{(\frac{p\,q}{p-q})^*(x)}dx < \infty$ \,\,\, for every $a>1$.
    \item The inclusion \,$ \lpv \hookrightarrow \lqv$\, is $L$-weakly compact.
    \item The inclusion $ \lpv \hookrightarrow \lqv$\, is $M$-weakly compact.
    \item The inclusion $ \lpv \hookrightarrow \lqv$\, is $DSS$.
    \item The restriction of  the inclusion \,$\lpv \hookrightarrow \lqv$\, on  any subspace spanned by a disjoint sequence \,$(\frac{\chi_{E_n}}{\mu(E_n)^{\frac{1}{p(t)}}})$\,  is  not an isomorphism.
    \item $\lim_{x\rw \mu(\Omega)^{-}} \, \,(\mu(\Omega)-x)^{(\frac{p-q}{p\,q})^*(x)}= 0$.
\end{enumerate}
\end{thm}

\begin{proof}
$(1)\Rightarrow (2)$. Let us define the exponent $s(t):= \frac{p(t)\,q(t)}{p(t)-q(t)}$.
  Then $\,1/q(t) = 1/p(t) +1/s(t)$\, and using H\"older norm inequality 
  we have
  $$\|f\,\chi_{A} \|_{q(\cdot)} \,\leq \, 2 \, \|f\|_{p(\cdot)}\, \|\chi_{A}\|_{\frac{p\,q}{p-q}(\cdot) }$$
Hence by the hypothesis and Proposition 3.2 we get
$$\lim_{\mu(A_{n})\rightarrow 0} \sup_{f\in B_{\lpv}} \|f \chi_{A_{n}}\|_{q(\cdot)} \, \leq \, 2 \lim_{\mu(A_{n})\rightarrow 0} \|\chi_{A_{n}}\|_{\frac{p\,q}{p-q}(\cdot)} \, = \, 0,$$
so the inclusion $\lpv \hookrightarrow \lqv$  is $L$-weakly compact.

  $(2) \Rightarrow (3) \Rightarrow (4) \Rightarrow (5)  $  are clear  (see  Theorem 4.1).

   $(5) \Rightarrow (6)$ is Proposition \ref{prop-33}.

   $(6) \Rightarrow (1)$.  Consider  the function \,$f(t) = \frac{p-q}{p\,q}(t)$. From   Lemma \ref{lema-a} we have that the functions \, $\frac{1}{\frac{p-q}{pq}(t)}=\frac{pq}{p-q}(t) $ \, and \,  $\frac{1}{(\frac{p-q}{pq})^*(x)} $ \, are equi-distributed. Thus \, $(\frac{p\,q}{p-q})^* = (1/f^*)^*$\, and
$$
\int_{0}^{\mu(\Omega)} a^{\left(\frac{pq}{p-q}\right)^*(x)}dx  = \int_{0}^{\mu(\Omega)} a^{\left(\frac{1}{(\frac{p-q}{pq})^*}\right)^*(x)}dx
$$
  Now, as $1/f^*$ is an increasing function  it follows, by  Lemma \ref{lema-2},  that  \,$(1/f)^*(x)= (1/f^*)^*(x) = (1/f^*)(\mu(\Omega)-x)$. Hence
$$
 = \,\int_{0}^{\mu(\Omega)} a^{\frac{1}{(\frac{p-q}{pq})^*}(\mu(\Omega)-x)}dx  \,=\,\int_{0}^{\mu(\Omega)} a^{\frac{1}{(\frac{p-q}{pq})^*(y)}}dy.$$
Finally, the  boundedness of this integral follows now  from the hypothesis and Lemma  \ref{prop-1} for  $\mu(\Omega) = b $,  since \, $(\frac{p-q}{pq})^*(\cdot)$ is decreasing. 
\end{proof}

Notice that above extends the $L$-weak compactness criterion in \cite{E-G-N} (Thm 3.4) including now unbounded exponents.

\begin{ex} \label{examples}

\begin{enumerate}
             \item Let \, $p(t) = 1 + \ln(1 -\ln t ) $\, on $(0,1)$. The inclusion  \, $ L^{\infty}(0,1) \hookrightarrow L^{p(\cdot)}(0,1) $ is strictly singular, since for every $a>1$\,
             $$\int_{0}^1 a^{1+\ln(1-\ln t) } dt < \infty.
             $$
             \item Let \, $p_{\alpha}(t) =\frac{1}{t^{\alpha}}$\, for $\alpha > 0$ on $(0,1)$. The inclusions  \,$L^{\infty}(0,1) \hookrightarrow L^{p_{\alpha}(\cdot)}(0,1) $ \, and  $  L^{p_{\alpha}(\cdot)}(0,1)  \hookrightarrow  L^{1}(0,1)$\, are  not DSS.

            The inclusions \, $L^{p_{\alpha}(\cdot)}(0,1)\hookrightarrow L^{p_{\beta}(\cdot)}(0,1)$,  for $ 0<\alpha <\beta \leq 1 $    or   $ 1\leq \beta <\alpha <\infty$, are  not DSS.



             \item Let \, $p_{\alpha}(t) =\frac{1}{(1-t^{\alpha})}$\, on $(0,1)$\, and  \,$\alpha>0$. The  inclusions\,\, $  L^{p_{\alpha}(\cdot)}(0,1)  \hookrightarrow  L^{1}(0,1) $\,  and \,$L^{\infty}(0,1) \hookrightarrow L^{p_{\alpha}(\cdot)}(0,1) $\,
                    are not weakly compact.

   \item Let \,$p_{\alpha}(t) = \ln^{\alpha}(\frac{1}{t})$ for \,$t \in (0,1/e)$  and  $0 < \alpha <\infty$. The inclusion \, $ L^{\infty}(0,1/e) \hookrightarrow  L^{p_{\alpha}(\cdot)}(0,1/e)$\,  is strictly singular  if and only if  \,   $0 < \alpha < 1$. Indeed, this follows from Theorem \ref{Teor-4},  since   
            $$ \lim_{x\rightarrow 1/e} \, \big( \frac{1}{e}-x\big)^{\frac{1}{\ln^{\alpha}(\,(\frac{1}{e}-x)^{-1}\,)}}\,= \,  0 $$
            if and only if  \,  $0 < \alpha < 1$.
       On the other side  the inclusions  \,$L^{p_{\alpha}(\cdot)}(0, 1/e) \hookrightarrow L^{1}(0, 1/e) $, for $\alpha > 0$, are not DSS.


     The inclusions \,$ L^{p_{\beta }(\cdot)}(0, 1/e) \hookrightarrow  L^{p_{\alpha }(\cdot)} (0,1/e)$  for  \,$0<\alpha < \beta <\infty$  are not  DSS.

  \item Let $\alpha >0$, $t \in (0,1/e)$ and
   $$p_{\alpha}(t) : =  \frac{ \ln^{\alpha}(\frac{1}{t})}{\ln^{\alpha}(\frac{1}{t}) -1}.
   $$  The inclusion \, $ L^{p_{\alpha}(\cdot)}(0,1/e)    \hookrightarrow \, L^{1}(0, 1/e) $\, is weakly compact if and only if \,$0 < \alpha < 1$.

   \item Let  $q(t): = \sqrt{\ln (\frac{1}{t})}$\, for $t\in (0, \frac{1}{e})$ and \, $p(t):= (1+\epsilon)q(t)$, for some $\epsilon >0$. The inclusion
       \, $ L^{p(\cdot)}(0,\frac{1}{e})  \hookrightarrow    L^{q(\cdot)}(0, \frac{1}{e})
       $\,  is DSS. Indeed,
       $$ \lim_{x\rw \frac{1}{e}} (\frac{1}{e} - x)^{(\frac{p-q}{pq})^*(x)} = \lim_{x\rw \frac{1}{e}} \left[(\frac{1}{e} - x)^{(\frac{1}{q})^*(x)} \right]^{\frac{\epsilon}{1+\epsilon}} = 0.
       $$
 \end{enumerate}
\end{ex}


In case of \,$ p^+ <\infty$, the condition \,$ess \inf (p-q)(\cdot) > 0$\, is a sufficient condition for the inclusion  $\lpv \hookrightarrow \lqv$ be  DSS (\cite{F-H-R-S} Prop.3.3) or $L$-weakly compact (\cite{E-G-N} Thm. 3.4). However this fails when \,$p^{+}= \infty$, see Example \ref{ejemplo-0} (and compare it with Example \ref{examples} (6)).


\section{Regular exponents}

In this section we  assume some  regularity for the exponents,  getting then a simpler DSS criterion.
Recall  that an scalar function $f$ over a metric measure space $(\Omega,\mu,d)$ is said to be (locally) \textit{log-H\"older continuous} if there exists a constant $C>0$ such that, for all $x\not =y\in\Omega$,
$$
\vert f(x)-f(y)\vert\leq \frac{C}{\ln\left(e+\frac{1}{d(x,y)}\right)}.
$$
The class of log-H\"older continuous  exponents  is very useful  in applications of variable exponent spaces (cf. \cite{Libro,Libro2}).



\begin{prop}\label{reordenada-log-Holder}
Let \,$f:[0,1]\rightarrow [0,\infty)$ be a log-H\"older continuous function. Then its decreasing rearrangement $f^*$ is also log-H\"older continuous.
\end{prop}

\begin{proof}
Assume that \,$f^*$ is not log-H\"older continuous, so there exist two sequences $(x_n)$ and $(y_n)$ in $[0,1]$ such that, for every natural $n$,
$$
\vert f^*(x_n)-f^*(y_n)\vert> \frac{n}{\ln\left(e+\frac{1}{|x_n-y_n|}\right)}.
$$ We can suppose \, $f^*(x_n) > f^*(y_n)$ \,for every $n$.
We will look for  sequences $(a_n)$ and $(b_n)$ such that \,$|a_n-b_n| \leq \vert x_n-y_n\vert$ \,and\,  $\vert f(a_n)-f(b_n)\vert\geq \vert f^*(x_n)-f^*(y_n)\vert$, since then
$$
\vert f(a_n)-f(b_n)\vert
\geq\vert f^*(x_n)-f^*(y_n)\vert
> \frac{n}{\ln\left(e+\frac{1}{\vert x_n-y_n\vert}\right)}
\geq \frac{n}{\ln\left(e+\frac{1}{| a_n-b_n)}\right)}.
$$

We know by the properties of  $f^*$ that
$$
\left|\Big\{t\in[0,1]:f^*(y_n) <f^*(t)< f^*(x_n) \,\Big\}\right|
\leq
\vert x_n-y_n\vert,
$$
and
\begin{equation}\label{equi2}
 \left|\Big\{t\in [0,1]: \,f^*(y_n)\}<f(t)<f^*(x_n)\Big\}\,\right|
\leq \, \vert x_n-y_n\vert.
\end{equation}

For every natural $n$ we define the disjoint compact sets \,$A_n:= f^{-1}([0,f^*(y_n)])$ and \,$B_n:= f^{-1}([f^*(x_n), \infty))$. Thus  there exist $a_n\in A_n$ and $b_n\in B_n$ such that
$$ |a_n-b_n| \, =\, \min \{|r-s| \,:\, r\in A_n  \text{ and } s\in B_n\}.
$$ Since $f$ is continuous we have
$$ f(a_n) \leq f^*(y_n) \quad \text{and}\quad f(b_n)\geq f^*(x_n).
$$ Now, for $t = \lambda a_n + (1-\lambda)b_n$ with $0<\lambda <1$, it follows from the definition of $a_n$ and $b_n$ that
$$ f^*(y_n) < f(t) < f^*(x_n).
$$ Using (\ref{equi2}), we conclude that
\,$  |a_n-b_n| \leq |x_n-y_n|.$
\end{proof}



\begin{prop}\label{DSS-log-Holder-separados}
Let  $p(\cdot)\geq q(\cdot)$ log-H\"older continuous exponents on $[0,1]$. If
$$
\lim_{x\rightarrow 1^{-}}\,(1-x)^{(p-q)^*(x)}=0,
$$
then \, \,$
ess\inf(p-q)>0.
$
\end{prop}

\begin{proof}
Suppose that $ess\inf(p-q)=0$ (hence  $\lim_{x\rightarrow 1}(p-q)^*(x)=0$).
Then, given $x_n\in[0,1)$, we can take $x_{n+1}$ close enough to $1$ to get
$$
(x_{n+1}-x_n)^{(p-q)^*(x_{n+1})}\simeq(1-x_n)^{(p-q)^*(1)}=(1-x_n)^0=1.
$$
Concretely, we take $x_{n+1}$ so that $(x_{n+1}-x_n)^{(p-q)^*(x_{n+1})}\geq\frac{1}{2}$. So, by induction, for every $x_0\in[0,1)$, we can construct a sequence $(x_n)\nearrow 1$ satisfying
$$
(x_{n+1}-x_n)^{(p-q)^*(x_{n+1})}\geq\frac{1}{2}.
$$

But, on the other side, $p(\cdot)$ and $q(\cdot)$ are log-H\"older continuous, so $(p-q)(\cdot)$ and $(p-q)^*(\cdot)$ are log-H\"older too by above proposition. If we also suppose that $\lim_{x\rightarrow 1}(1-x)^{(p-q)^*(x)}=0$ we reach a contradiction, as
\begin{align*}
   (x_{n+1}-x_n)^{(p-q)^*(x_{n+1})}
   ={}  &   (x_{n+1}-x_n)^{(p-q)^*(x_{n+1})-(p-q)^*(x_n)+(p-q)^*(x_n)}\\
   ={}  &   \left(\frac{1}{x_{n+1}-x_n}\right)^{(p-q)^*(x_n)-(p-q)^*(x_{n+1})}\cdot(x_{n+1}-x_n)^{(p-q)^*(x_n)}\\
   \leq{}   &    \left(\frac{1}{x_{n+1}-x_n}\right)^{\frac{M}{\ln\left(e+\frac{1}{x_{n+1}-x_n}\right)}}\cdot(1-x_n)^{(p-q)^*(x_n)}\xrightarrow{n\rightarrow\infty} 0,
\end{align*}
because $(1-x_n)^{(p-q)^*(x_n)}\xrightarrow{n\rightarrow\infty} 0$\, and \,$\left(\frac{1}{x_{n+1}-x_n}\right)^{\frac{M}{\ln\left(e+\frac{1}{x_{n+1}-x_n}\right)}}\rightarrow e^M. $
\end{proof}




\begin{cor}\label{equiv-DSS-log-Holder}
Let \,$p(\cdot)\geq q(\cdot)$ \, be log-H\"older continuous exponents on $[0,1]$ with $p^+<\infty$. TFAE:
\begin{enumerate}
    \item \,$ess\inf (p(\cdot)-q(\cdot)) > 0$.
    \item The inclusion \,$\lpvo\hookrightarrow \lqvo$\,  is DSS.
    \item $\lim_{x\rightarrow 1^{-}}\,\,(1-x)^{(p-q)^*(x)}=0$.
\end{enumerate}
\end{cor}

\begin{proof}
   The above proposition shows  that $(1)$ and $(3)$ are equivalent,  and the others follows from Theorem \ref{Teor-2}.
\end{proof}
   
\begin{rem}
    Notice that every other statement at Theorem \ref{Teor-2} is also equivalent.
\end{rem}


\begin{ex}
(i) Inclusions \,$L^{p_{\alpha}(\cdot)}[0,1]\hookrightarrow L^{p}[0,1]$\, are  not DSS ,  for \, $p_{\alpha}(x)=p+x^{\alpha}$, \, $\alpha>0$\, and \, $1\leq p < \infty$.
\end{ex}

In general \,$ess\inf (p-q)(\cdot)>0$\, is not an equivalence for an inclusion  be DSS but just a sufficient condition (even in the case of  be $p(\cdot)$  continuous   and $q(\cdot)$  log-H\"older continuous):

(ii) Take a log-H\"older continuous exponent  $q(\cdot)$,  the continuous function
$$
r(x)=\frac{\ln\left(\left[\log_2(1-x)\right]^{2j}\right)}{-\log_2(1-x)},
$$
(for $j$  natural) and the exponent \,$p(\cdot)=q(\cdot)+r(1-2^{-e})\chi_{[0,1-2^{-e})}+r(\cdot)\chi_{[1-2^{-e}, 1]}
$. Then \,$ess\inf (p-q)(\cdot)=0$\,  but the inclusion \,$L^{p(\cdot)}[0,1] \hookrightarrow L^{q(\cdot)}[0,1]$\, is $M$-weakly compact for big enough $j\geq p^+$ and hence DSS (see \cite{H-R-S-21} p.9).

We do not know whether above criteria can be extended to bounded open subsets in \,$R^{n}$\, ($n\geq 2$).


\section{The infinite measure case}






In order to study DSS inclusions on an infinite measure, we can always assume that \,$\mu(\Omega_{d}) = \infty$\, (in Proposition \ref{inclusion}), avoiding  trivial cases of non-DSS inclusions.

If the inclusion \,$\lpv\hookrightarrow \lqv$\,  holds and $p^+<\infty$, then the sets\, $D_\varepsilon$\, (for every $\varepsilon>0$) has infinite measure, where
$$
D_\varepsilon:=\{t\in \Omega_d: p(t)<q(t)+\varepsilon\}.
$$


Indeed,  assume that there exists $\varepsilon>0$ such that $\mu(D_\varepsilon)<\infty$. Then, \, $\mu(\Omega_d\setminus D_\varepsilon)=\infty$\, and  if  \, $r(t): = \frac{p\,q}{p-q}(t)$, we have \,$r^+_{\vert \Omega_d\setminus D_\varepsilon}\leq\frac{p^+ q^+}{\varepsilon}<\infty$. \, Hence, for every $\lambda>1$,
$$
\int_{\Omega_d} \lambda^{-r(t)} d\mu \geq \int_{\Omega_d\setminus D_\varepsilon} \lambda^{-r(t)} d\mu \geq \lambda^{-\frac{p^+ q^+}{\varepsilon}} \mu(\Omega_d\setminus D_\varepsilon)=\infty
$$
which, using  Proposition \ref{inclusion}, gives a contradiction.



\begin{prop}\label{DSS-inf}
Let $(\Omega,\mu)$ be an atomless infinite measure space and exponents \,$p(\cdot)$ and $q(\cdot)$.  If the inclusion \,$\lpv \hookrightarrow \lqv$\, holds and \,$ p^{+} <\infty$, then the inclusion is non-DSS.
\end{prop}

\begin{proof}
We shall proceed in a similar way as in \cite{H-R-S-22} Thm 3.4. It is enough to find a disjoint sequence $(f_n)$ generating the same infinite dimensional (closed) subspace in $\lpv$ as well as in $\lqv$. To do so, if we could take functions \, $f_n:=\chi_{A_n}$, where $(A_n)$ is a disjoint sequence with  $\mu(A_n)=1$ verifying that \, \,$p^+_{\vert_{A_n}}-q^-_{\vert_{A_n}}<\frac{1}{n}$, the proof would be finished.

Indeed, under these hypotheses we have the inclusions
$$
\ell_{p^-_{\vert_{A_n}}}\hookrightarrow [f_n]_{p(\cdot)} \hookrightarrow \ell_{p^+_{\vert_{A_n}}}
$$
and
$$
\ell_{q^-_{\vert_{A_n}}}\hookrightarrow [f_n]_{q(\cdot)} \hookrightarrow \ell_{q^+_{\vert_{A_n}}}.
$$
Using now Proposition \ref{lemaNa} we have that the Nakano sequence spaces \, $\ell_{p^-_{\vert_{A_n}}} \cong \ell_{p^+_{\vert_{A_n}}}\cong \ell_{q^-_{\vert_{A_n}}} \cong \ell_{q^+_{\vert_{A_n}}}$. Hence, we deduce that \,  $[f_n]_{p(\cdot)}\cong\, [f_n]_{q(\cdot)}$.

Let us now construct such disjoint  sequence $(A_n)$ with the above properties. We remarked above that for every natural $n$\, we have 
$\mu(D_{\frac{1}{2n}})=\infty$ \, (or either $p(\cdot)\vert_{A}\equiv q(\cdot)\vert_{A}$ over a positive measure subset $A\subset \Omega$ and then the proof is trivial). Thus, if we make a finite partition \,$\{[x_i, x_{i+1})\}$ of the interval \,$[1,p^+)$\, where \,$1=x_1<x_2<...<x_k= p^+$\, and \,$x_{i+1}-x_i<\frac{1}{2n}$\,  for every natural $i$, we can assure that, for some $j$,
$$
\mu\left(D_{\frac{1}{2n}}\cap p^{-1}\left([x_j,x_{j+1})\right)\right)=\infty.
$$
Even more, by  the definition of \,$D_{\frac{1}{2n}}$, it is also true that
$$
\mu\left(D_{\frac{1}{2n}}\cap p^{-1}\left([x_j,x_{j+1})\right) \cap q^{-1} ([x_{j}-\frac{1}{2n}, x_{j+1}) ) \right)=\infty.
$$
Thus, we conclude that there exists a set $E_{n}$ (the above set) with  infinite measure such that
\, $$p^+_{E_n}-q^{-}_{E_n}\leq x_{j+1}-(x_{j}-\frac{1}{2n})\, \leq \, \frac{1}{n} .$$
 Now, since $\mu(E_n)= \infty$\,  for each natural $n$, we can take \,$A_n\subset \left(E_n\cap (\bigcup_{i=1}^{n-1} A_i)^c\right)$\,  with $\mu(A_n)=1$ getting so the needed disjoint sequence $(A_n)$.
\end{proof}

\begin{prop}\label{DSS-Inf2} Let $(\Omega, \mu)$ be an atomless infinite  measure space and exponents \,$p(\cdot)$ and \,$q(\cdot)$. If the inclusion \,$\lpv \hookrightarrow \lqv$\, holds and the distribution function of \,$q(\cdot)$\, verifies \,$\mu_{q(\cdot)}(n) = \infty$ \,for every natural $n$, then the inclusion is not DSS.
\end{prop}

\begin{proof}
Consider the sets
$$ D_n= \{\,t\in \Omega \,\,:\,\,q(t) > n\,\} \, \subset \, E_n = \{\,t\in \Omega \,\,:\,\,p(t) > n\,\}.
$$
 Since \,$\mu(D_n)= \infty$, the increasing  sequence \,$(D_n\backslash E_m)_{m}$\, verifies \,$\mu(D_n \backslash E_m) \nearrow \infty$  as $m\rw\infty$.

 Hence, by the hypotheses, it is possible to find a strictly increasing  sequence \,$(n_k)_k $ \, of natural numbers and a  disjoint  sequence of measurable sets \,$(A_k)_{k}$\,  with \,$\mu(A_k) = 1$ \, such that
$$
 n_k \leq q^-_{|A_k} \leq p^-_{|A_k} \qquad \text{and} \qquad  q^+_{|A_k} \leq p^+_{|A_k} \leq n_{k+1}.
$$
Consider now  the (closed) subspace \,$[\chi_{A_k}]_{p(\cdot)}$, formed by  all functions of the form \,$\sum_{k} \lambda_{k}\chi_{A_{k}} \in \lpv$\, such that $\rho_{p(\cdot)}(\lambda\sum_{k>N} \lambda_{k}\chi_{A_{k}})\xrightarrow{{N\rw\infty}} 0$ for every $\lambda > 0$.

If \,$\sum_k \lambda_k \chi_{A_k} \in \lpv$, then,  for some $r>0$,

$$
\sum_k |\frac{\lambda_k}{r}|^{n_{k+1}}\leq  \sum_k |\frac{\lambda_k}{r}|^{p^+_{|A_k}} \leq \rho_{p(\cdot)}\big(\frac{\sum_k \lambda_k\chi_{A_k}}{r}\big) \leq \sum_k |\frac{\lambda_k}{r}|^{p^-_{|A_k}}\leq \sum_k |\frac{\lambda_k}{r}|^{n_k}.
$$
Hence, since the Nakano sequence spaces \, $\ell_{(n_k)} \cong \ell_{\infty}  \cong \ell_{(n_{k+1})}$, we deduce that the subspace \,$ [\chi_{A_k}]_{p(\cdot)}  \cong c_0$.

Analogously   the corresponding subspace defined similarly  \,$[\chi_{A_k} ]_{q(\cdot)}$\,   in $\lqv$   satisfies   that  \, $ [\chi_{A_k}]_{q(\cdot)} \cong  c_0$. Therefore we conclude that \, $ [\chi_{A_k}]_{p(\cdot)} \cong [\chi_{A_k}]_{q(\cdot)} \cong c_0$.
\end{proof}

Combining now the two above propositions we cover the different cases for the  the following  statement (notice that the case \, $q^{+}<p^{+}= \infty $\, and $\mu_{p(\cdot)}(n) = \infty$ for every natural $n$ is not possible according with Proposition \ref{inclusion}).

\begin{thm} \label{DSS-Inf3}
Let $(\Omega,\mu)$ be a non-atomic infinite measure space and exponents \,$p(\cdot)$ and $q(\cdot)$. If the inclusion \,$\lpv \hookrightarrow \lqv$\, holds, then it is not DSS.
\end{thm}

\begin{proof}
In the case of the existence of a subset \,$\Omega_0 \subset \Omega$ \, with \,  $\mu(\Omega_0) = \infty$\, and \, $p^+_{|\Omega_0} < \infty$, the  above Proposition \ref{DSS-inf} gives the result.  For the other possible cases we can use now Proposition \ref{DSS-Inf2}.
\end{proof}


\begin{rem}
    The  non-DSS property of inclusions between variable Lebesgue spaces on infinite measures  is not true in the class of Orlicz spaces. For example  the inclusions \, $L^{\varphi}(0,\infty) \hookrightarrow L^{\psi}(0,\infty)$\, are DSS  for the Orlicz functions \,$\varphi(x)=x^r\vee x^s$\, and \,$\psi(x)=x^p \vee x^q$, ($ r <p \leq q< s $) (cf. \cite{G-H-R} Ex. 4.9).
\end{rem}

\begin{rem} Note that also for infinite measures if the inclusion  \, $L^{\infty}(\mu) \hookrightarrow \lpv$  \, holds then it  is  non-DSS.

Indeed, assume that for an exponent function $p(\cdot)$  the inclusion \,  $L^{\infty}(\mu) \hookrightarrow \lpv$ \, holds. Then, \,  $\chi_{\Omega} \in \lpv $,  so  there exists \, $\lambda>1$\, such that \,$\int_{\Omega} \frac{d\mu}{\lambda^{p(t)}} \, < \infty$. It follows that  the sets \,$D_{n}=\{t\in \Omega: p(t)> n\}$\, has infinite measure for every natural $n$ (since \, $\mu(D_{n}^{c}) < \infty)$. Now reasoning as in above Proposition {\ref{DSS-Inf2}}, we can easily deduce  that  the inclusion  \,$L^{\infty}(\mu) \hookrightarrow \lpv$\, is non-DSS.

\end{rem}


\bibliographystyle{amsplain}

\begin{thebibliography}{10}



















































\bibitem{A-K} F. Albiac and N.J. Kalton; \emph{Topics in Banach space theory} (2nd ed.). Springer, 2016.

\bibitem{A-B} Ch. Aliprantis and O. Burkinshaw; \emph{Positive operators}. Academic Press, 1985.

\bibitem{A} S.V. Astashkin; \emph{The Rademacher system in function spaces}. Birkh\"auser, 2020.

\bibitem{A-H-S} S.V. Astashkin, F.L. Hern\'andez and E.M. Semenov; \emph{Strictly singular inclusions of rearrangement invariant spaces and Rademacher spaces.} Studia Math. \textbf{193} (2009), 269-283.

\bibitem{A-S} S.V. Astashkin and E.M. Semenov; \emph{Some properties of embeddings of rearrangement invariant spaces.} Sbornik Math. \textbf{210:10} (2019), 1361-1379.

\bibitem{B-S} C. Bennett and R.C. Sharpley; \emph{Interpolation of operators.} Pure and Applied Mathematics \textbf{129}. Academic Press, 1988.

\bibitem{Libro2} D.V. Cruz-Uribe and A. Fiorenza; \emph{Variable Lebesgue spaces: foundations and harmonic analysis.} Birkh\"auser/Springer, 2013.

\bibitem{Libro} L. Diening, P. Harjulehto, P. H\"{a}st\"{o} and M. Ru\v{z}i\v{c}ka; \emph{Lebesgue and Sobolev spaces with variable exponents.} Lecture Notes in Mathematics, vol. \textbf{2017}. Springer, 2011.

\bibitem{E-G-N} D.E. Edmunds, A. Gogatishvili and A. Nekvinda; \emph{Almost compact and compact embeddings of variable exponent spaces.} Studia Math. \textbf{268} (2023), no. 2, 187-211.

\bibitem{E-L-N} D.E. Edmunds, J. Lang and A. Nekvinda; \emph{On $L^{p(x)}$ norms.} Proc. R. Soc. London A \textbf{255} (1999), 219-225.

\bibitem{F-G-N-R} A. Fiorenza, A. Gogatishvili, A. Nekvinda and J. Rokotoson; \emph{Remarks on compactness results for variable exponent spaces $L^{p(\cdot)}$.} J. Math. Pures Appl. \textbf{157} (2022), 136-144.

\bibitem{F-H-K-T} J. Flores, F.L. Hern\'andez, N.J. Kalton and P. Tradacete; \emph{Characterizations of strictly singular operators on Banach lattices.} J. London Math. Soc. (2) \textbf{79} (2009), no. 3, 612-630.

\bibitem{F-H-R-S} J. Flores, F.L. Hern\'andez, C. Ruiz and M. Sanchiz; \emph{On the structure of variable exponent spaces.} Indag. Math. (N.S.) \textbf{31} (2020), no. 5, 831-841.

\bibitem{G-H-R} A. Garc\'ia del Amo, F.L. Hern\'andez and C. Ruiz; \emph{Disjointly strictly singular operators and interpolation.} Proc. Roy. Soc. Edinburgh Sect. A \textbf{126} (1996), no. 5, 1011-1026.

\bibitem{G-H-S-S} A. Garc\'ia del Amo, F.L. Hern\'andez, E. Semenov and V. S\'anchez; \emph{Disjointly strictly singular inclusions between rearrangement invariant spaces.} J. London Math. Soc. \textbf{62} (2000), 239-252.

\bibitem{Go} S. Goldberg; \emph{Unbounded linear operators: theory and applications.} Dover, 1966.

\bibitem{G-M} P. G\'orka and A. Macios; \emph{Almost everything you need to know about relatively compact sets in variable Lebesgue spaces.} J. Funct. Anal. \textbf{269} (2015), 1925-1949.

\bibitem{H-N-S} F.L. Hern\'andez, S.Ya. Novikov and E.M. Semenov; \emph{Strictly singular embeddings between rearrangement invariant spaces.} Positivity \textbf{7}:1-2 (2003), 119-124.

\bibitem{H-R-S-12} F.L. Hern\'andez, Y. Raynaud and E.M. Semenov; \emph{Bernstein widths and superstrictly singular inclusions.} Oper. Theory Adv. Appl. vol. \textbf{218} (2012), 359-376.

\bibitem{H-S} F.L. Hern\'andez and B. Rodr\'iguez-Salinas; \emph{On $\ell_p$-complemented copies in Orlicz spaces II.} Israel J. Math. \textbf{68} (1989), 27-55.

\bibitem{H-S2} \underline{\hspace{12em}}; \emph{Lattice embedding $L^p$ into Orlicz spaces.} Israel J. Math. \textbf{90} (1995), 167-188.

\bibitem{H-R-S-21} F.L. Hern\'andez, C. Ruiz and M. Sanchiz; \emph{Weak compactness in variable exponent spaces.} J. Funct. Anal. \textbf{281} (2021), 109087.

\bibitem{H-R-S-22} \underline{\hspace{12em}}; \emph{Weak compactness and representation in variable exponent Lebesgue spaces on infinite measure spaces.} Rev. Real Acad. Cienc. Exactas Fis. Nat. Ser. A-Mat, \textbf{116},152 (2022).

\bibitem{K-R} M. Krasnoselskii and Ya. Rutickii; \emph{Convex functions and Orlicz spaces.} Noordhoff, 1961.

\bibitem{K-P-S} S. Krein, J. Petunin and E. Semenov; \emph{Interpolation of linear operators.} Transl. Amer. Math. Soc., 1982.

\bibitem{LT1} J. Lindenstrauss and L. Tzafriri; \emph{Classical Banach spaces I.} Springer-Verlag, 1977.

\bibitem{LT2} \underline{\hspace{12em}}; \emph{Classical Banach spaces II.} Springer-Verlag, 1979.


\bibitem{M-W} L. Maligranda and W. Wnuk; \emph{Landau-type theorems for variable Lebesgue spaces.} Commentationes Math. \textbf{55} (2015), 119-126.

\bibitem{Mu} J. Musielak; \emph{Orlicz spaces and modular spaces.} Lecture Notes in Mathematics, vol. \textbf{1034}. Springer, 1983.

\bibitem{N} H. Nakano; \emph{Modulared sequence spaces.} Proc. Japan Acad. \textbf{27} (1951), 411-415.

\bibitem{Tesis} M. Sanchiz; \emph{Structure and operators on variable Lebesgue spaces.} Ph.D. Thesis, Madrid Complutense University, 2023.






\end{thebibliography}

\providecommand{\bysame}{\leavevmode\hbox
to3em{\hrulefill}\thinspace}
\providecommand{\MR}{\relax\ifhmode\unskip\space\fi MR }
\providecommand{\MRhref}[2]{%
  \href{http://www.ams.org/mathscinet-getitem?mr=#1}{#2}
} \providecommand{\href}[2]{#2}

\end{document}